\documentclass{amsart}
\usepackage{amsmath,amssymb,hyperref,mathrsfs,color}
\usepackage{paralist}
\usepackage{calc}
 \newtheorem{The}{Theorem}[section]

 \newtheorem{Pro}[The]{Proposition}
 \theoremstyle{definition}
 \newtheorem{defn}[The]{Definition}
 \theoremstyle{remark}
 \newtheorem{Rem}[The]{Remark}
 \newtheorem{ex}{Example}
 \numberwithin{equation}{section}
\newcommand{\T}{\mathbb{T}}
\newcommand{\R}{\mathbb{R}}
\newcommand{\Z}{\mathbb{Z}}
\newcommand{\N}{\mathbb{N}}

\title{Homoclinic orbits and critical points \\of barrier functions}
\author{Piermarco Cannarsa \and Wei Cheng}
\address{Dipartimento di Matematica, Universit\`a di Roma Tor Vergata,
Via della Ricerca Scientifica 1, 00133 Roma, Italy}
\email{cannarsa@mat.uniroma2.it}
\address{Department of Mathematics, Nanjing University,
Nanjing 210093, China}
\email{chengwei@nju.edu.cn}
\date{\today}
\begin{document}
\maketitle

\begin{abstract}
We interpret the close link between the critical points of Mather's barrier functions and  minimal homoclinic orbits with respect to the Aubry sets on $\T^n$. We also prove a critical point theorem for  barrier functions, and the existence of such homoclinic orbits on $\T^2$ as an application.
\end{abstract}

\section{Introduction}
In the huge literature that is devoted to the study of homoclinic orbits of Hamiltonian systems and dates back, at least, to the works of Poincar\'e, one can single out one important approach which is based on the Ambrosetti-Rabinowitz critical point theory see, e.g., \cite{Bolotin},\cite{Bolotin_book},\cite{Bolotin-Rabinowitz},\cite{Rabinowitz-Tanaka}.   Another powerful viewpoint in this context is provided by Mather's theory~\cite{Mather91,Mather93} and weak KAM theory (see, e.g. Fathi's book \cite{Fathi-book}) which helped to clarify many aspects of minimal orbits and invariant sets, see, e.g., \cite{Fathi98},\cite{Bernard 2000},\cite{Contreras-Paternain},\cite{Cui-Cheng-Cheng},\cite{Zheng-Cheng},\cite{Zhou}.

In this paper, we adopt a mixed strategy to investigate  the existence of minimal homoclinic orbits with respect to the Aubry sets for a given Tonelli Hamiltonian, using  critical point theory for certain barrier functions. 
We concentrate on the case of the $n$-torus throughout the paper, even if some of our results can be proved for more general manifolds using similar ideas. 

Let $H$ be a Tonelli Hamiltonian on $\T^n$, and consider the associated Hamilton-Jacobi equation which has the form
$$
H(x,c+Du(x))=\alpha(c),\qquad x\in\T^n,
$$
where $c\in\R^n$ stands for a cohomology class in $H^1(\T^n,\R)$ and  $\alpha(\cdot)$ is Mather's function. For fixed $c$, under the generic condition that the Aubry class is unique, it is easy to define the barrier function $B_c^*$
 as difference of two weak KAM solutions forming a conjugate pair $(u^-_c,u^+_c)$, that is, 
$$
B_c^*(x)=u^-_c(x)-u^+_c(x),\qquad x\in\T^n.
$$
As is well-known for viscosity solutions, $B^*_c$ turns out to be a locally semiconcave function with linear modulus, see \cite{Cannarsa-Sinestrari}. Moreover, in \cite{Cannarsa-Cheng-Zhang}, it was proved  that $x$ determines a homoclinic orbit with respect to the projected Aubry set $\mathscr{A}_c$ whenever $x$ is a critical point of $B^*_c$ outside $\mathscr{A}_c$ and the limiting differentials $D^*u^-_c(x)$ and $D^*u^+_c(x)$ have a nonempty intersection. More precisely, one can show that there exists a $C^2$ extremal curve $\gamma:(-\infty,\infty)\to\T^n$ such that $\gamma(0)=x$ and the $\alpha$- and $\omega$- limit sets of $\gamma$ belong to $\mathscr{A}_c$ even if $x$ is a singular critical point of $B^*_c$.

Building on the above result, our construction of minimal homoclinic (even heteroclinic) orbits is obtained in two steps:
\begin{enumerate}[(1)]
\item we need find enough critical points outside $\mathscr{A}_c$, and
\item we need a criterion to ensure that such critical points can indeed create minimal homoclinic orbits. 
\end{enumerate}
For the first step, we prove the following.
\begin{The}
Let $L$ be a Tonelli Lagrangian on $\T^n$, and for fixed $c\in\R^n$, suppose the projected Aubry set $\mathscr{A}_c$ consists of a single Aubry class.  Then there exist at least $\text{Cat}(\T^n\setminus U)$ critical points of the barrier function $B^*_c$ outside $\mathscr{A}_c$, where $U$ is any  open neighborhood of $\mathscr{A}_c$.
\end{The}
\noindent
In the above statement  $\text{Cat}(\T^n\setminus U)$ stands for the Lusternik-Schnirelmann category  of  $\T^n\setminus U$. Moreover, we recall that the uniqueness of the Aubry class holds for a generic Tonelli Hamiltonian as explained in Section~\ref{se:pre} below.

For the second step, we give the following criterion
where $\Lambda^+_x$ denotes the superlevel set of $B^*_c$ at $x$.
\begin{The}
Let $x\in\R^n$ be a critical point of $B^*_c$. Then we have that $$0\in D^{\ast}{u^-_c}(x)-D^{\ast}{u^+_c}(x)$$ if any of the following conditions is satisfied:
\begin{enumerate}[(a)]
  \item at least one of the two solutions $v^-$ and $v^+$ is differentiable at $x$,
  \item the tangent space to $\Lambda^+_x$ at $x$ is such that
\begin{equation}\label{non-degen_intro}
\mbox{\rm dim}\big(\mbox{\rm Tan}(x,\Lambda^+_x)\big)\geqslant n-1,
\end{equation}
\item $n=2$ and there exists a unit vector $\theta$ such that $\langle p,\theta\rangle\geqslant 0$ for all $p\in D^+B^*_c(x)$,
\item $n=2$ and $x$ is not an isolated critical point of $B^*_c$.\end{enumerate}
\end{The}
Since a $B^*_c$ is at most semiconcave, critical points have to be interpreted and dealt with in a nonsmooth setting. For this purpose, we borrow a result from Attouch~\cite{Attouch} which applies the Lasry-Lions regularization method to critical points, see Proposition \ref{same_critical} below. Appealing to the above theorems, we obtain the following. 

\begin{The}\label{intro:main_thm}
Let $L$ be a Tonelli Lagrangian on $\T^2$, and for fixed $c\in\R^2$, suppose the projected Aubry set $\mathscr{A}_c$ consists of a single Aubry class.  If there exists an open neighborhood $U\subset\T^2$ of $\mathscr{A}_c$ such that $\T^2\setminus U$ is non-contractible, then there exists a minimal homoclinic orbit with respect to the Aubry set, $\tilde{\mathscr{A}}_c$, lying outside $\tilde{\mathscr{A}}_c$.
\end{The}

Is is worth noting that our approach also works to construct connecting orbits among distinct Aubry classes if the projected Aubry set has more than one class. 

\begin{The}
Let $L$ be a Tonelli Lagrangian on $\T^2$, and for fixed $c\in\R^2$, suppose the projected Aubry set $\mathscr{A}_c$ consists of  finitely many Aubry classes. If there exists an open neighborhood $U\subset\T^2$ of $\mathscr{A}_c$ such that $\T^2\setminus U$ is non-contractible, then there must be a connecting orbit between any pair of distinct Aubry classes.
\end{The}

The paper is organized as follows. In Section 2,  we recall basic material on  semiconcave functions and weak KAM theory. In section 3, we discuss the homoclinic phenomenon and critical points of semiconcave functions, and prove the main results of the paper. 

\bigskip

\noindent{\bf Acknowledgments} This work was partially supported by the Natural Scientific Foundation of China (Grant No. 11271182), the  National Basic Research Program of China (Grant No. 2013CB834100), and the National Group for Mathematical Analysis and Probability of the Italian Istituto Nazionale di Alta Matematica ``Francesco Severi''. The  authors are grateful to Albert Fathi, Ludovic Rifford and Antonio Siconolfi for helpful discussions and comments on the results of this paper.

\section{Preliminary facts}
\label{se:pre}
\subsection{Hamilton-Jacobi equations and viscosity solutions}
Let $\T^n$ be the $n$-dimensional torus. We denote by $T\T^n$ the tangent bundle of $\T^n$ and by $T^{\ast}\T^n$ the cotangent bundle.

\begin{defn}
A function $L:T\T^n\to\R$ is said to be a {\em Tonelli Lagrangian} if the following assumptions are satisfied.
\begin{enumerate}[(L1)]
  \item {\em Smoothness}: $L=L(x,v)$ is of class at least $C^2$.
  \item {\em Convexity}: The Hessian $\frac{\partial^2 L}{\partial v^2}(x,v)$ is positive definite on each fibre $T_x\T^n$.
  \item {\em Superlinearity}:
$$
\lim_{|v|\to\infty}\frac{L(x,v)}{|v|}=\infty\quad\text{uniformly for } x\in\T^n.
$$
\end{enumerate}
\end{defn}

Given a Tonelli Lagrangian  $L$,  the {\em Tonelli Hamiltonian} $H=H(x,p)$ associated with  $L$ is defined as follows:
$$
H(x,p)=\max\big\{\langle p,v\rangle-L(x,v): v\in T_x\T^n\big\},\quad (x,p)\in T^{\ast}\T^n\,.
$$
It is easy to see that for any Tonelli Lagrangian $L$, the associated Hamiltonian $H$ satisfies  similar smoothness ($H$ is of class at least $C^2$), convexity, and superlinearity conditions, which will be referred to as (H1), (H2), and (H3).

Throughout this paper we will be concerned with the Hamilton-Jacobi equation
\begin{equation}\label{H_J}
H_c(x,Du(x))=H(x,c+Du(x))=\alpha(c)\qquad (x\in \T^n)
\end{equation}
with $H$ any Tonelli Hamiltonian and $H_c(x,\cdot):=H(x,c+\cdot)$, where $c\in\R^n$ and $\alpha:\R^n\to\R$ is Mather's $\alpha$-function. In other words, we can suppose that $H=H(x,p)$ is $\Z^n$-periodic in the $x$ variable,  convex and superlinear in the $p$ variable, and $u$ is a $\Z^n$-periodic solution of \eqref{H_J}.

We say that $u:\T^n\to\R$ is a {\em viscosity subsolution} (resp. {\em supersolution}) of \eqref{H_J}, if for each $C^1$ function $\phi: M\to\R$ such that $u-\phi$ admits a maximum (resp. a minimum) at $x\in M$, we have
$$
H_c(x,D\phi(x))\leqslant\alpha(c),\quad (\text{resp.}\ H_c(x,Du(x))\geqslant\alpha(c)).
$$
We say that $u: \T^n\to\R$ is a {\em viscosity solution}, if it is both a subsolution and a supersolution. A viscosity solution of \eqref{H_J} is called a {\em critical} viscosity solution.

\subsection{Semiconcave functions}

Let $\Omega\subset\R^n$ be open and convex. A function $u:\Omega\to\R$ is {\em semiconcave} (with linear modulus) if there exists a constant $C>0$, such that
$$
\lambda u(x)+(1-\lambda)u(y)-u(\lambda x+(1-\lambda)y)\leqslant\frac C2\lambda(1-\lambda)|x-y|^2
$$
for any $x,y\in\Omega$ and $\lambda\in[0,1]$, and $C$ is called a {\em semiconcavity constant} for $u$ in $\Omega$. A function $u:\Omega\to\R$ is said to be {\em locally semiconcave} if for any $x\in\Omega$, there exists an open convex subset $U\ni x$ such that $u|_U$ is semiconcave.

Let $u:\Omega\subset\R^n\to\R$ be a semiconcave function, for any $x\in\Omega$, the set
\begin{align*}
D^-u(x)&=\left\{p\in\R^n:\liminf_{y\to x}\frac{u(y)-u(x)-\langle p,y-x\rangle}{|y-x|}\geqslant 0\right\},\\
D^+u(x)&=\left\{p\in\R^n:\limsup_{y\to x}\frac{u(y)-u(x)-\langle p,y-x\rangle}{|y-x|}\leqslant 0\right\}.
\end{align*}
are called the (Dini) {\em subdifferential} and {\em superdifferential} of $u$ at $x$ respectively.

\begin{Pro}[\cite{Cannarsa-Sinestrari}]\label{criterion-Du_semiconcave}
Let $u:\Omega\to\R$ be a function on $\Omega\subset\R^n$. If there exists a constant $C>0$ such that, for any $x\in\Omega$, there exists $p\in\R^n$ such that
\begin{equation}\label{criterion_for_lin_semiconcave}
u(y)\leqslant u(x)+\langle p,y-x\rangle+\frac C2|y-x|^2,\quad \forall y\in\Omega,
\end{equation}
then $u$ is semiconcave with constant $C$.

If $u$ is semiconcave function on $\Omega\subset\R^n$ with constant $C$, then \eqref{criterion_for_lin_semiconcave} holds for any $p\in D^+u(x)$.
\end{Pro}

Let $u:\Omega\to\R$ be locally Lipschitz. A vector $p\in\R^n$ is called a {\em limiting differential} of $u$ at $x$ if there exists a sequence $\{x_n\}\subset\Omega\setminus\{x\}$ such that $u$ is differentiable at $x_k$ for each $k\in\N$, and
$$
\lim_{k\to\infty}x_k=x\quad\text{and}\quad \lim_{k\to\infty}Du(x_k)=p.
$$
The set of all limiting differentials of $u$ at $x$ is denoted by $D^{\ast}u(x)$.

\begin{Pro}\label{basic_facts_of_superdifferential}
Let $u:\Omega\subset\R^n\to\R$ be a semiconcave function and $x\in\Omega$. Then the following properties hold.
\begin{enumerate}[(a)]
  \item $D^+u(x)$ is a nonempty compact  convex set in $\R^n$ and $D^{\ast}u(x)\subset\partial D^+u(x)$, where $\partial D^+u(x)$ denotes the topological boundary of $D^+u(x)$.
  \item The set-valued function $x\rightsquigarrow D^+u(x)$ is upper semicontinuous.
  \item If $D^+u(x)$ is a singleton, then $u$ is differentiable at $x$. Consequently, if $D^+u(x)$ is a singleton for every points in $\Omega$, then $u\in C^1$.
  \item $D^+u(x)$ equals the convex hull of $D^{\ast}u(x)$.
  \item $D^{\ast}u(x)=\{\lim_{i\to\infty}p_i: p_i\in D^+u(x_i), x_i\to x,\mathrm{diam}(D^+u(x_i))\to 0\}$.
\end{enumerate}
\end{Pro}

A point $x\in\Omega$ is called a {\em singular point} of $u$ if $D^+u(x)$ is not a singleton. The set of all singular points of $u$, also called the {\em singular set} of $u$, is denoted by $\Sigma_u$.

\subsection{Facts from weak KAM theory}
In what follow, $H$ stands for a Tonelli Hamiltonian on the $n$-torus $\T^n$ and $L$ for the corresponding Tonelli Lagrangian.

\begin{defn}
Let $L$ be a $C^2$ Tonelli Lagrangian on $T\T^n$ and set, for any $c\in\R^n$,
$$L_c(x,v)=L(x,v)-\langle c,v\rangle \qquad\forall (x,v)\in T\T^n.$$ A function $u_c:\T^n\to\R$ is said to be {\em dominated} by $L_c+\alpha(c)$ iff, for each absolutely continuous arc $\gamma:[a,b]\to\T^n$ with $a<b$, one has
$$
u_c(\gamma(b))-u_c(\gamma(a))\leqslant\int^b_aL_c(\gamma(s),\dot{\gamma}(s))ds+\alpha(c)(b-a).
$$
When this happens, one writes $u_c\prec L_c+\alpha(c)$.
\end{defn}

\begin{defn}
Let $c\in\R^n$, and $u_c$ be a real-valued function on $\T^n$. An absolutely continuous curve $\gamma:[a,b]\to\T^n$ is said to be $(u_c,L_c,\alpha(c))$-{\em calibrated} if
$$
u_c(\gamma(b))-u_c(\gamma(a))=\int^b_aL_c(\gamma(s),\dot{\gamma}(s))ds+\alpha(c)(b-a).
$$
\end{defn}

The following well known (see, e.g. \cite{Fathi-book} \cite{Fathi-Siconolfi2004}) facts are useful to clarify the relation between viscosity solutions and weak KAM solutions.

\begin{Pro}
Let $c\in\R^n$. A function $u_c:\T^n\to\R$ is dominated by $L_c+\alpha(c)$ if and only if $u_c$ is a viscosity subsolution of \eqref{H_J}.

If $u_c$ is a viscosity solution of \eqref{H_J}, then there exists an absolutely continuous arc $\gamma_x:(-\infty,0]\to \T^n$ with $\gamma_x(0)=x$ such that $\gamma_x$ is $(u_c,L_c,\alpha(c))$-calibrated.
\end{Pro}

Now, we recall some semiconcavity properties of  viscosity solutions. The following result is fundamental (see, e.g., \cite{Fathi-book} \cite{Rifford}).

\begin{Pro}
Any viscosity solution of the Hamilton-Jacobi equation \eqref{H_J} is locally semiconcave with linear modulus.
\end{Pro}
The following is the weak KAM analogue of \cite[Theorem~6.4.12]{Cannarsa-Sinestrari}
\begin{Pro}\label{Ext_and_reachable}
$\mathrm{Ext}\,D^+u(x)=D^{\ast}u(x)$ for any viscosity solution $u$ of \eqref{H_J} and any $x\in\T^n$.
\end{Pro}
Finally, we recall a result which connects calibrated curves with limiting differentials (see \cite{Cannarsa-Sinestrari} and \cite{Rifford}).
\begin{Pro}\label{reachable_grad_and_backward}
Let $x\in \T^n$ and $u:\T^n\to\R$ be a viscosity solution of the Hamilton-Jacobi equation \eqref{H_J}. Then $p\in D^{\ast}u(x)$ if and only if there exists a $C^1$ curve $\gamma:(-\infty,0]\to \T^n$ with $\gamma(0)=x$ which is $(u,L_c,\alpha(c))$-calibrated, and $p=\frac{\partial L_c}{\partial v}(x,\dot{\gamma}(0))$.
\end{Pro}

\subsection{Barrier functions}
For $t>0$, $x,y\in\T^n$ and $c\in\R^n$, we introduce the following
quantity
\begin{equation}
h^c_t(x,y)=\inf\int^t_0L_c(\gamma(s),\dot{\gamma}(s))\ ds,
\end{equation}
where the infimum is computed over all absolutely continuous arcs $\gamma:[0,t]\to\T^n$ such that
$\gamma(0)=x$ and $\gamma(t)=y$.

Let $c\in\R^n$ and let $h^c_t(x,y)$ be defined as above. {\em Peierls' barrier} is defined as
\begin{equation}\label{Peierls'-barrier}
h_c(x,y)=\liminf_{t\to\infty}h^c_t(x,y)+\alpha(c)t.
\end{equation}
We call $\mathscr{A}_c=\{x\in\T^n~:~h_c(x,x)=0\}$ the {\em projected Aubry set}. It is well known that $\mathscr{A}_c$ is nonempty and compact for any $c\in\R^n$.

\begin{Pro}\label{h_c_determin_viscosity_solution}
{\em (\cite{Fathi-Siconolfi2004})} If  Peierls' barrier $h_c$ is finite then, for each
  $x\in\T^n$, $u_c(y):=h_c(x,y)$ is a global critical viscosity
  solution of \eqref{H_J}.
  Moreover, for any $x,y\in\T^n$, there is an arc $\xi:(-\infty,0]\to\T^n$, with
  $\xi(0)=y$, such that
  $$
  u_c(\xi(0))-u_c(\xi(-t))=\int^0_{-t}L_c(\xi(s),\dot{\xi}(s))\ dt+\alpha(c)t,\quad \forall\,t\geqslant0\,.
  $$
\end{Pro}

Let $c\in\R^n$ and let $h_c$ be Peierls' barrier. The {\em barrier function} $B_c^{\ast}(x)$ is defined by Mather (\cite{Mather93}) as
\begin{equation}\label{defn_barrier}
B_c^{\ast}(x)=\inf_{y,z\in\mathscr{M}_c}\{h_c(y,x)+h_c(x,z)-h_c(y,z)\},\quad x\in\T^n,
\end{equation}
where $\mathscr{M}_c$ is the projected Mather set, that is, the projection onto $\T^n$ of  Mather's set $\tilde{\mathscr{M}}_c$. Note that $\mathscr{M}_c\subset\mathscr{A}_c$  (see, e.g., \cite{Bernard2002}\cite{Mather91}\cite{Mather93}).
By Proposition~\ref{h_c_determin_viscosity_solution}, $h_c(x,\cdot)$ gives a global viscosity solution of \eqref{H_J} and  $h_c(\cdot,x)$  a global critical solution of \eqref{H_J} with the Hamiltonian $\breve{H}(x,p)=H(x,-p)$. Fix $y,z\in\mathscr{M}_c$ and, for each $x\in\T^n$, let
$$
u_{c,y}^-(x)=h_c(y,x),\quad u_{c,z}^+(x)=-h_c(x,z).
$$
Then
\begin{equation}\label{B_c_star}
B^{\ast}_c(x)=\inf_{y,z\in\mathscr{M}_c}\{u_{c,y}^-(x)-u_{c,z}^+(x)-h_c(y,z)\}.
\end{equation}

For any $x,y\in\T^n$, define  {\em Mather's pseudometric} (see \cite{Mather93}) on $\mathscr{A}_c$ by
\begin{equation*}
d_c(x,y)=h_c(x,y)+h_c(y,x)\,.
\end{equation*}
Two points  $x,y\in\mathscr{A}_c$ are said to be in the same {\em Aubry class} if $d_c(x,y)=0$.

\begin{Pro}\label{uniqueness}
Let $x,y\in\mathscr{A}_c$ be distinct points in the same Aubry class. Then $h_c(x,\cdot)$ equals $h_c(y,\cdot)$ up to a constant. If $x,y\in\mathscr{A}_c$, $x\ne y$, belong to different Aubry classes, then $h_c(x,\cdot)-h_c(y,\cdot)$ is not constant.
\end{Pro}

From Proposition \ref{uniqueness}, it follows that each Aubry class $A$ determines---up to constants---a viscosity solution of the form $h_c(y,x)$ for any $y\in A$. Now, suppose there exists a finite number of Aubry classes $A_1,\ldots,A_k$, and denote by $u_i$, $i=1,\ldots,k$, the corresponding uniquely determined viscosity solutions. Each of such solutions is called an {\em elementary weak KAM solution}\footnote{The concept of elementary weak KAM solution was introduced by Chong-Qing Cheng in an alternative way  in \cite{CCQ_Proc}, see also \cite{Fathi-Figalli-Rifford}.}. It is not hard to show that if there exists a unique Aubry class, then we can represent the barrier function $B^{\ast}_c$ in the form
\begin{equation}\label{u+u-}
B^{\ast}_c(x)=u_{c,y}^-(x)-u_{c,y}^+(x):=u_c^-(x)-u_c^+(x)\,.
\end{equation}
In this case, $(u_c^-,u_c^+)$  is called a conjugate pair of weak KAM solutions (see \cite{Fathi-book}).

Recall that $\mathcal{S}_-$ usually denotes the set of all viscosity solution $u_c^-$ of the Hamilton-Jacobi equation
\begin{equation}\label{eq:weak_KAM_eq}
H(x,c+Du(x))=\alpha(c).
\end{equation}
Setting $\breve{H}(x,p)=H(x,-p)$, it is clear that $\breve{H}$ is also a Tonelli Hamiltonian. Let us denote by $\mathcal{S}_+$ the set of  all viscosity solutions, $-u_c^+$,   of the corresponding Hamilton-Jacobi equation. Then $(u_c^-,u_c^+)$ is  a {\em conjugate pair of weak KAM solutions} if $u_c^-(x)=u_c^+(x)$ for any $x\in\mathscr{M}_c$.

Let us consider a conjugate pair $(u_c^-,u_c^+)$ of weak KAM solutions. We denote by $\mathcal{I}(u^-,u^+)$, the set
$$
\mathcal{I}(u_c^-,u_c^+)=\{x\in\T^n: u_c^-(x)=u_c^+(x)\}.
$$
We have $\mathcal{I}(u_c^-,u_c^+)\supset\mathscr{M}_c$. Under the assumption that there exists a unique Aubry class, it is easy to see that $\mathcal{I}(u_c^-,u_c^+)=\mathscr{A}_c$. In other words
$$
\mathscr{A}_c=\{x\in\T^n: B^*_c(x)=0\}.
$$

A set $\mathscr{L}$ of Tonelli Lagrangians is said to be {\em generic} (in the sense of Ma\~n\'e) if there exists a residual\footnote{Recall that, in a complete metric space, a subset is called residual if it is  the intersection of a countable  family of dense open subsets.} set $\mathcal{O}\subset C^2(\T^n)$ and a Tonelli Lagrangian $L_0$ such that  each $L\in\mathscr{L}$ has the form
\begin{equation*}
L=L_0+V
\end{equation*}
for some $V\in\mathcal{O}$. A similar notion can be given for a set of Tonelli Hamiltonians.

Examples of generic properties of interest to this paper are the following:
\begin{equation}\label{generic_condition}
\text{there exists a unique Aubry class in $\mathscr{A}_c$ for fixed $c$,}\tag{GC1}
\end{equation}
and
\begin{equation}\label{generic_condition_2}
\text{there exists a finite number of Aubry classes in $\mathscr{A}_c$ for all $c$.}\tag{GC2}
\end{equation}
Indeed, a well-known result by Ma\~n\'e~\cite{Mane} ensures that \eqref{generic_condition} holds for a generic family of Tonelli Hamiltonians. Consequently, for any fixed $c\in\R^n$, there is a unique viscosity solution of the equation associated with any Hamiltonian  of such a generic family. It is also known that \eqref{generic_condition_2} is a generic property (\cite{Bernard-Contreras}). 
In this case, for all $c\in \R^n$,  there exists a finite number of elementary weak KAM solutions.

It is well known that $u_c^-$ (resp. $u_c^+$) is a locally semiconcave (resp. semiconvex) function with linear modulus. Then the barrier function $B^{\ast}_c$ is also a locally seminconcave function with linear modulus, see, e.g., \cite[Proposition 2.1.5]{Cannarsa-Sinestrari}. Given any conjugate  pair $(u^-_c,u^+_c)$ of weak KAM solutions, one can lift the problem to the universal covering space $\R^n$ defining
\begin{equation}\label{v}
v^-(x)=u^-_c(x)+\langle c,x\rangle,\quad v^+(x)=u^+_c(x)+\langle c,x\rangle,\quad x\in \R^n.
\end{equation}
Then, under the generic condition \eqref{generic_condition}, we have
$$
B^*_c(x)=v^-(x)-v^+(x),\quad x\in \R^n.
$$

It is worth noting that both $D^*v^-(x)$ and $D^*v^+(x)$ are contained in the corresponding energy surface, i.e
\begin{equation}\label{eq:level_set_H}
\begin{split}
H(x,p)&=\alpha(c),\quad p\in D^*v^{\pm}(x),\\
H(x,p)&<\alpha(c),\quad p\in D^{\mp}v^{\pm}(x)\setminus D^*v^{\pm}(x).
\end{split}
\end{equation}
Indeed, the former assertion of \eqref{eq:level_set_H} follows directly from the definition of $D^*v^{\pm}(x)$ and the fact that the equation holds at all points of differentiability. In order to justify the latter, one just need to combine the inclusions
$$D^{\mp}v^{\pm}(x)\subset\{p:H(x,p)\leqslant\alpha(c)\}$$
with the property $\mathrm{Ext}\,D^{\mp}v^{\pm}(x)=D^*v^{\pm}(x)$ (see Proposition \ref{Ext_and_reachable}) and the strict convexity of $H(x,\cdot)$.

\section{connecting orbits and critical points of barrier functions}

\subsection{A criterion on Homoclinic orbits}
We call $x\in\R^n$ a (generalized) {\em critical point} of a locally semiconcave function $u$ if $0\in D^+u(x)$. Moreover, $x$ is called a critical point of {\em saddle type} if $0\in D^+u(x)$ and $x$  is not a local minimum or maximum point of $u$.

In \cite{Cannarsa-Cheng-Zhang}, we proved the following criterion for the existence of homoclinic orbits with respect to the Aubry set under a certain condition on limiting differentials.
\begin{Pro}\label{homoclinic_orbit}
Let $x\in\Sigma_{B^{\ast}_c}$, and let $B^{\ast}_c(x)=u^-_c(x)-u^+_c(x)$ where $(u^-_c,u^+_c)$ is a conjugate pair of weak KAM solutions. If
\begin{equation}\label{nonempty}
D^{\ast}{u^-_c}(x)\cap D^{\ast}{u^+_c}(x)\not=\varnothing,
\end{equation}
then there exists a minimal homoclinic orbit with respect to the Aubry set $\tilde{\mathscr{A}}_c$ passing through $x$.
\end{Pro}

It is clear that condition \eqref{nonempty} is equivalent to 
\begin{equation}\label{0_in_reachable}
0\in D^{\ast}{u^-_c}(x)-D^{\ast}{u^+_c}(x)=D^{\ast}{v^-}(x)-D^{\ast}{v^+}(x),
\end{equation}
where $v^\pm$ is defined in \eqref{v}. Moreover, if \eqref{0_in_reachable} is satisfied, then $x$ must be a critical point of $B^*_c$. Notice that the fact that $x$ is a singular point of $B^*_c$ is inessential here.

\begin{ex}
Let $L(x,v)=\frac 12|v|^2-(1-\cos x)$ be a one-dimensional pendulum system. For $c=0$, $\mathscr{A}_0=\{2k\pi\}$, and $x_k=(2k+1)\pi$, $k\in\Z$ are singular points of the unique weak KAM solution $u^-_0$ up to constants. It is clear that $u^+_0=-u^-_0$, and the barrier function $B^*_0=2u^-_0$. Since $D^+u^-_0(x_k)=D^-u^+_0(x_k)=[-2,2]$, we have $0\in D^*u^-_0(x_k)-D^*u^+_0(x_k)$, and the two types of separatrices give the expected homoclinic orbits.
\end{ex}

\begin{defn}
A vector $\theta\in\R^n$ belongs to the {\em contingent cone} (or {\em Bouligand's tangent cone}) $T_S(x)$ iff there exist sequences $\theta_i\in \R^n$, converging to $\theta$, and $t_i\in\R^+$, decreasing to $0$, such that
$$
x+t_i\theta_i\in S\,,\quad \forall i\geqslant 1\,.
$$
%A vector $\theta\in\R^n$ belongs to {\em Clarke's tangent cone}   $C_S(x)$ iff, for all sequences $x_i\in S$ converging to $x$ and  $ t_i\in\R^+$ decreasing to $0$, there is a sequence  $\theta_i\in \R^n$
%$$
%x_i+t_i\theta_i\in S\,,\quad \forall i\geqslant 1\,.
%$$
The vector space generated by $T_S(x)$ is called the {\em tangent space} to $S$ at $x$ and is denoted by $\text{Tan}(x,S)$.
\end{defn}
%Note that $C_S(x)$ is always contained in $T_S(x)$.

We define the superlevel set of $B^*_c$ with respect to a given $x\in\R^n$ as
$$
\Lambda^+_x=\{y\in\R^n: B^*_c(y)\geqslant B^*_c(x)\}.
$$
The following criterion gives sufficient conditions for \eqref{0_in_reachable}  to hold true.
\begin{The}\label{criterion_I}
Let $x\in\R^n$ be a critical point of $B^*_c$. Then we have that $$0\in D^{\ast}{u^-_c}(x)-D^{\ast}{u^+_c}(x)$$ if any of the following conditions is satisfied:
\begin{enumerate}[(a)]
  \item at least one of the two solutions $v^-$ and $v^+$ is differentiable at $x$,
  \item the tangent space to $\Lambda^+_x$ at $x$ is such that
\begin{equation}\label{non-degen}
\mbox{\rm dim}\big(\mbox{\rm Tan}(x,\Lambda^+_x)\big)\geqslant n-1,
\end{equation}
\item $n=2$ and there exists a unit vector $\theta$ such that $\langle p,\theta\rangle\geqslant 0$ for all $p\in D^+B^*_c(x)$,
\item $n=2$ and $x$ is not an isolated critical point of $B^*_c$.\end{enumerate}
\end{The}

\begin{Rem}\label{re:MP}
\rm Notice that, when $n=2$, condition $(c)$ above is satisfied whenever $x$ is not a local maximum point of $B^*_c$.
\end{Rem}
%Observe that condition $(c)$ above holds true (when $n=2$) if 
%$x$ is not a local maximum point of $B^*_c$ and
\begin{proof}
Let $x\in\R^n$ and let $0\in D^+{u^-_c}(x)-D^-{u^+_c}(x)=D^+v^-(x)-D^-v^+(x)$.

First, suppose condition (a) holds. Without loss of generality, we can assume that $v^+$ is differentiable at $x$ with $Dv^+(x)=p^+$, i.e., $Dv^+(x)=\{p^+\}$. Then, $D^+B^*_c(x)=D^+v^-(x)-Dv^+(x)$ by the sum rule for the superdifferential of concave functions and $p^+\in D^+v^-(x)$ because $x$ is a critical point of $B^*_c$. Now, 
\begin{equation}\label{eq:energy_surface}
Z_{x,E}:=\{p\in T_x^*M: H(x,p)\leqslant E\}
\end{equation}
is a nonempty compact convex set and $\partial Z_{x,E}$ is $C^2$ smooth or a singleton, under the energy condition $E=\alpha(c)$. 
Since $p^+\in \partial Z_{x,E}$, we conclude that $$p^+\in D^+v^-(x)\cap\partial Z_{x,E}= D^*v^-(x).$$
So, there exists $p^-\in D^*v^-(x)$ such that $p^-=p^+$, or $0\in D^*v^-(x)-D^*v^+(x)$.

We now assume condition (b). Then there exist linearly independent unit vectors $\{\theta^i_x\}_{i=1}^{n-1}\subset T_{\Lambda^+_x}(x)$.
%, $|\theta^i_x|=1$, $i=1,\ldots,n-1$ 
By the semiconcavity of $B^*_c$, for every $i=1,\dots,n-1$ there exists a sequence $\{x^i_k\}\subset\Lambda^+_x$, converging to $x$  as $k\to\infty$, such that
$$
B^*_c(x)\leqslant B^*_c(x^i_k)\leqslant B^*_c(x)+\langle p,x^i_k-x\rangle+\frac C2|x^i_k-x|^2,\quad\forall p\in D^+B^*_c(x).
$$
This implies that there exist  limiting vectors $\theta^i$ of $\{(x^i_k-x)/|x^i_k-x|\}$,such that
\begin{equation}\label{eq:half_space}
\langle\theta^i, p\rangle\geqslant 0,\quad\forall p\in D^+B^*_c(x)\quad(i=1,\dots,n-1).
\end{equation}
Consequently, there exist $\lambda_i\in\R\;(i=1,\dots,n-1)$ such that
%$$
%\min_{p^-\in D^+v^-(x)}\langle\theta^i,p^-\rangle\geqslant\lambda_i\geqslant\max_{p^+\in D^-v^+(x)}\langle\theta^i,p^+\rangle,
%$$
%and
%\begin{align*}
%\langle\theta_x, p^-\rangle&\geqslant\lambda,\quad p^-\in D^+v^-(x)\\
%\langle\theta_x, p^+\rangle&\leqslant\lambda,\quad p^+\in D^-v^+(x).
%\end{align*}

%By \eqref{non-degen} and a similar argument above, there exist linear independent $\{\theta^i_x\}_{i=1}^{n-1}\subset\R^n$, $|\theta^i_x|=1$, $i=1,\ldots,n-1$, and $\{\lambda_i\}_{i=1}^{n-1}\subset\R$, such that for each $i\in\{1,\ldots,n-1\}$,
\begin{align*}
\langle\theta^i, p^-\rangle&\geqslant\lambda_i,\quad \forall\,p^-\in D^+v^-(x)\\
\langle\theta^i, p^+\rangle&\leqslant\lambda_i,\quad \forall\, p^+\in D^-v^+(x).
\end{align*}
Let $\ell$ be the intersection of the hyperplanes $\Pi_i=\{p: \langle\theta^i,p\rangle=\lambda_i\}$, $i=1,\ldots,n-1$, and observe that  $\ell$ is a straight line because $\{\theta^i\}_{i=1}^{n-1}$ are linearly independent. Since $0\in D^+B^*_c(x)$, there exist covectors $p^-\in D^+v^-(x)$ and $p^+\in D^-v^+(x)$ such that $p^-=p^+$. If $p^-\not\in D^*v^-(x)$, then there exist  $p^-_1,p^-_2\in\text{Ext}(\ell\cap D^+v^-(x))$
such that $p^-$ is in the interior of the line segment $[p_1^-,p_2^-]$\footnote{Recall that,  given any convex set $C\subset\R^n$ and supporting hyperplane $H$, the extremal points of $C\cap H$ are still extremal points of $C$ (see, e.g., \cite[Lemma 2.7.1]{Borwein-Vanderwerff}).}.  We also have $\ell\cap\partial Z_{x,E}=\{p^-_1,p^-_2\}$ since all the extremal points of $D^+v(x)$ are contained in $\partial Z_{x,E}$. Similarly, if $p^+\not\in D^*v^+(x)$, then there exist covectors $p^+_1,p^+_2\in\ell\cap\text{Ext} (D^-v^+(x))$ such that $\ell\cap\partial Z_{x,E}=\{p^+_1,p^+_2\}$. This implies that $\text{Ext} (D^+v^-(x))\cap \text{Ext} (D^-v^+(x))\not=\varnothing$. So, recalling the equality $\text{Ext} (D^{\pm}v^{\mp}(x))=D^*v^{\mp}(x)$ once again (see Proposition~\ref{Ext_and_reachable}), we conclude that $0\in D^*v^-(x)-D^*v^+(x)$.

Next, observe that  condition (c) is just a special case of (b).

Finally, suppose $x$ is not an isolated critical point of $B^*_c$ as in condition (d). Then there exists a sequence of critical points $y_j$ converging to $x$. By the semiconcavity of $B^*_c$,  for any $p\in D^+B^*_c(x)$ and $p_j\in D^+B^*_c(y_j)$ we have that
\begin{align*}
B^*_c(y_j)&\leqslant B^*_c(x)+\langle p,y_j-x\rangle+\frac C2|y_j-x|^2\\
B^*_c(x)&\leqslant B^*_c(y_j)+\langle p_j,x-y_j\rangle+\frac C2|y_j-x|^2.
\end{align*}
Choosing $p_j=0$ for all $j$ since each $y_j$ is a critical point of $B^*_c$, and combining the two inequalities above, we have
$$
0\leqslant\langle p,y_j-x\rangle+C|y_j-x|^2,\quad\forall p\in D^+B^*_c(x).
$$
This means there exists a unit vector $\theta$ which satisfies condition (c).
%such that
%$$
%\langle p,\theta\rangle\geqslant 0,\quad\forall p\in D^+B^*_c(x).
%$$
%It follows that $0\in D^*u^-_c(x)-D^*u^+_c(x)$ by the same argument previously which completes the proof.
\end{proof}

A celebrated result in the theory of differential dynamical systems from the sixties is Smale's theorem on transversal homoclinic points which describes, in particular, the complicated dynamical behavior produced by Smale's horseshoe. When the Aubry set $\tilde{\mathscr{A}}_c$ is composed of hyperbolic fixed points or periodic orbits, the ``non-degenerate'' condition \eqref{non-degen} in Theorem \ref{criterion_I} is closely linked to how the unstable submanfold $\{(x,Dv^-(x)): x\in\R^n\}$ and the stable submanifold $\{(x,Dv^-(x)): x\in\R^n\}$ intersect.

In general, it is hard to tell whether a critical point $x$ of $B^*_c$ is a singular point or a regular one although, by semiconcavity, each local minimum point of $B^*_c$ must be regular. In the special case when $n=2$ and $B^*_c$ is of class $C^2$ in a neighborhood, $B(x,\varepsilon)$, of an isolated critical point $x$, condition \eqref{non-degen} yields the following dichotomy:
\begin{enumerate}[1)]
  \item $x$ {\em is a non-degenerate critical point of }$B^*_c$.

 \noindent In this case, the local unstable submanfold $\{(x,Dv^-(x)): x\in B(x,\varepsilon)\}$ and the local stable submanifold $\{(x,Dv^-(x)): x\in B(x,\varepsilon)\}$ intersect transversally, and it is clear that $\mbox{\rm dim}(\mbox{\rm Tan}(x,\Lambda^+_x))=n$.

  \item $x$ {\em is a degenerate critical point of $B^*_c$ such that $D^2B^*_c(x)$ has exactly one eigenvalue equal to 0 with one-dimensional eigenspace.}

\noindent  In this case, the phenomenon of homoclinic tangency may happen and 
%for the superlevel set $\Lambda^+_x$, 
$x$ can be a cusp point of the level set, with $\mbox{\rm dim}(\mbox{\rm Tan}(x,\Lambda^+_x))=n-1$.
\end{enumerate}
It is interesting to compare this analysis to the result in \cite{Pujals-Sambarino}, where some hyperbolic assumption on the limit sets of the homoclinic orbits is required.

\subsection{Lasry-Lions regularization}
It is clear that, in the case of $\T^n$, we can regard $u_c^\pm$, as well as $B=B^*_c=v^--v^+$, as $\Z^n$-periodic locally semiconcave functions on $\R^n$. Now we recall the regularization technique,  known as sup/inf convolution,  which is due to Lasry~and~Lions~\cite{Lasry-Lions}. A detailed formulation of this  method in the finite dimensional case can be found in \cite{Attouch}. 

%Instead of the standard one, we use the Tonelli Lagrangian itself which leads to more required properties for the applications to Hamiltonian dynamical systems. We need a further assumption on the Tonelli Lagrangian $L$,
%\begin{enumerate}[(L4)]
%  \item $L(x,v)$ is strictly convex with respect to $v$, i.e., there exists a constant $C_1>0$ such that $L_{vv}(x,v)\geqslant M_1I_n$, where $I_n$ is the identity, and there also exists a constant $C_2>0$ such that $L_x(x,v)/|v|^2\leqslant C_2$, $\forall v\in\R^n$ for $x$ uniformly.
%\end{enumerate}

%\begin{Rem}
%A typical Lagrangian satisfying (L4) is the mechanical system on $\T^n$, i.e.,
%$$
%L(x,v)=\frac12\langle A(x)v,v\rangle-V(x),\quad x\in\R^n, v\in\R^n,
%$$
%where $\langle A(x)v,v\rangle$ gives a $\T^n$-periodic Riemannian metric and $V$ is a $\T^n$-periodic $C^2$ potential. 
%\end{Rem}

For any semiconcave function $u^-:\R^n\to\R$, any semiconvex function $u^+:\R^n\to\R$, and any $\lambda>0$, we define
\begin{align}
u^-_{\lambda}(x)&=\sup_{y\in\R^n}\{u^-(y)-\frac 1{2\lambda}{|x-y|^2}\},\label{L-L regularity}\\
u^+_{\lambda}(x)&=\inf_{y\in\R^n}\{u^+(y)+\frac 1{2\lambda}{|x-y|^2}\}.
\end{align}
The following result characterizes the fundamental approximation properties of $u^\pm$ by $u^\pm_{\lambda}$ when $\lambda>0$ is small enough. For the reader's convenience, we provide a new proof of such properties below.

\begin{Pro}\label{same_critical}
Suppose $u^-:\R^n\to\R$ is a semiconcave function with constant $C$. Then, for every $0<\lambda\leqslant\lambda_0$, $0<\lambda_0<<1$, the function $u^-_{\lambda}$ in \eqref{L-L regularity} satisfies the following.
\begin{enumerate}[(P1)]
  \item $u^-_{\lambda}$ is of class $C^{1,1}(\R^n)$.
  \item As $\lambda\searrow 0$, $u^-_{\lambda}$ decreases to $u^-$ and $Du^-_{\lambda}\to D^+u^-$ in the graph sense.
\item $\lim_{\lambda\to0}Du^-_{\lambda}(x)=p_x$, where $p_x$ is the element of minimal norm of $D^+u^-(x)$.  
%$$
%\frac 12|p_x|^2=\min_{p\in D^+u(x)}\frac 12|p|^2,\quad\forall x\in\R^n.
%$$
  \item The functions $u^-$ and $u^-_{\lambda}$ have the same critical points and critical values when $\lambda<\min\{\lambda_0,C^{-1}\}$.
  \item In particular, there exists $0<\lambda_1\leqslant \lambda_0$ such that $u^-$ and $u^-_{\lambda}$ have the same local maximum points  when $\lambda\in (0,\lambda_1]$.
\end{enumerate}
\end{Pro}

\begin{proof} Hereafter, we drop the minus superscript and write simple $u, u_\lambda$ instead of $u^-, u^-_\lambda$.
It is worth noting that the definition of $u_{\lambda}$ in \eqref{L-L regularity} is actually a local one, that is the supremum, in fact the maximum, is taken in some ball $B(x,\rho)$, where $\rho$ only depends on $\lambda$ and $x$ in our case (see, e.g., \cite[Lemma 3.5.2]{Cannarsa-Sinestrari}).

Properties (P1) and (P2), for $0<\lambda<\lambda_0$, $0<\lambda_0<<1$, can be derived directly from \cite{Cannarsa-Sinestrari} except for the fact  that $Du_{\lambda}\to D^+u$ in the graph sense. This last property follows from the fact that the semiconcavity costant of $u_\lambda$ is uniform for $\lambda $ sufficiently small.

We proceed to prove (P3). For fixed $x\in\R^n$, let
$$
F(y,x)=u(y)-\frac 1{2\lambda}|x-y|^2 \qquad (y\in \bar{B}(x,\rho)) 
$$
and set
\begin{align*}
M(x)&=\{y\in \bar{B}(x,\rho): u_{\lambda}(x)=F(y,x)\}\\
Y(x)&=\{D_xF(y,x): y\in M(x)\}.
\end{align*}
We have that $D^+u_{\lambda}(x)=\text{co}\,Y(x)$, the convex hull of $Y(x)$ (see \cite[Theorem 3.4.4]{Cannarsa-Sinestrari}). If $0<\lambda\leqslant\lambda_0$ for $\lambda_0$ small enough, $u_{\lambda}$ is of class $C^{1,1}$ (see \cite[Theorem 3.5.3]{Cannarsa-Sinestrari}). In this case, it is clear that $Y(x)$ is a singleton, and so is $M(x)$. Set $M(x)=\{y_{\lambda}\}$ and $Y(x)=\{v_{\lambda}\}$ where $v_{\lambda}=(y_{\lambda}-x)/\lambda$. We note that $y_\lambda\in B(x,\rho)$ for $\lambda$ small enough.
Since $F(y,x)$ attains its maximum at $y=y_{\lambda}$, we have that $v_{\lambda}\in D^+u(y_{\lambda})$. By the semiconcavity of $u$, for any $p\in D^+u(x)$, we have
\begin{align*}
u(x)&\leqslant u(y_{\lambda})+\langle v_{\lambda}, x-y_{\lambda}\rangle+\frac C2|x-y_{\lambda}|^2\\
    &\leqslant u(x)+\langle p,y_{\lambda}-x\rangle+\langle v_{\lambda}, x-y_{\lambda}\rangle+C|x-y_{\lambda}|^2.
\end{align*}
Then,
\begin{equation}\label{basic_inequalty}
\langle p-v_{\lambda}, v_{\lambda}\rangle+\lambda C|v_{\lambda}|^2\geqslant 0,\quad \forall p\in D^+u(x).
\end{equation}
In view of \eqref{basic_inequalty}, it is easily checked that $\{v_{\lambda}\}$ is bounded when $0<\lambda\leqslant\lambda_0$. Without loss of generality, we suppose $v_{\lambda_k}\to v_0$ as $\lambda_k\to0$. So, taking the limit in \eqref{basic_inequalty} yields
$$
\langle p,v_0\rangle\geqslant\langle v_0,v_0\rangle,\quad \forall p\in D^+u(x).
$$
In other words,
%\begin{equation}\label{optimization_problem}
%\frac 12|p|^2\geqslant\frac 12|p_0|^2,\quad \forall p\in D^+u(x),
%\end{equation}
%where 
$p_0=v_0$ is the unique element of minimal norm of $D^+u(x)$.
%solve the associated optimization problem \eqref{optimization_problem}. 
Since $p_0$ is independent of the choice of $v_{\lambda_k}$, we have that $\lim_{\lambda\to0}v_{\lambda}=v_0$ and so $\lim_{\lambda\to0}y_{\lambda}=x$, which completes the proof of (P3).
% and the second part of (P2).

For the proof of (P4), note that if $x$ is a critical point of $u$, taking $p=0$ in \eqref{basic_inequalty} we have
$$
(\lambda C-1)|v_{\lambda}|^2\geqslant 0.
$$
It follows that $v_{\lambda}\equiv0$ for $0<\lambda<C^{-1}$, which means $x$ is also a critical point of $u_{\lambda}$. In this case, $y_{\lambda}\equiv x$ and $u_{\lambda}(x)=u(x)$  for $0<\lambda<C^{-1}$. Conversely, if $x$ is a critical point of $u_{\lambda}$, then $0=Du_{\lambda}(x)$, i.e., $\frac{y_\lambda-x}{\lambda}=v_\lambda=0$, which implies $y_\lambda=x$, and so, $0\in D^+u(x)$ and $u(x)=u_\lambda(x)$.

To prove (P5), we suppose $x_0$ is a local maximum point of $u$, i.e., $u(x_0)\geqslant u(x)$, for any $x\in B(x_0,\varepsilon)$. Using (P4), we get
$$
u_{\lambda}(x_0)=u(x_0)\geqslant u(z)\geqslant u(z)-\frac{|z-x|^2}{2\lambda},\quad z\in B(x_0,\varepsilon).
$$
When $\lambda>0$ is small enough such that the maximum of $u(\cdot)-\frac{|\cdot-x|^2}{2\lambda}$ is achieved in $B(x_0,\varepsilon)$. this implies $u_\lambda(x_0)\geqslant u_\lambda(x)$, for all $x\in B(x_0,\varepsilon)$. Conversely, if $x_0$ is a local maximum point of $u_\lambda$, i.e., $u_\lambda(x_0)\geqslant u_\lambda(x)$, for any $x\in B(x_0,\varepsilon)$. Then, Using (P4) again, we have
$$
u(x_0)=u_\lambda(x_0)\geqslant u_\lambda(x)\geqslant u(x),\quad x\in B(x_0,\varepsilon).
$$
\end{proof}

%\begin{Rem}
%By the compactness, there exists $\lambda_0$ such that properties (P1)-(P4) holds for any $\Z^n$-periodic locally semiconcave function $u$ and $\lambda<\lambda_0$.
%\end{Rem}

\subsection{Critical points of barrier functions}
Recalling the local semiconcavity of the barrier function $B^*_c(x)$, let $B_{\lambda}$ ($0<\lambda<\lambda_0$ small enough) be the corresponding Lasry-Lions regularization of $B^*_c(x)$ defined in \eqref{L-L regularity}. Then $B_{\lambda}$ has the same critical points as $B^*_c(x)$ by (P4) in Theorem \ref{same_critical}. If $x$ is a critical point of the barrier function $B^*_c(x)$, then $x$ produces  homoclinic orbits with respect to Aubry set $\tilde{\mathscr{A}}_c$ under any of the conditions of Theorem \ref{criterion_I}.

So, our first aim in this section is to look for critical points of the barrier function $B^*_c(x)$ outside the Aubry set, which is the set of the global minimizers of $B^*_c(x)$. For this purpose, we will use topological tools to obtain lower bounds for the  number of critical points of $B^*_c(x)$ outside the projected Aubry set.

Let $M$ be a closed  smooth $n$-dimensional manifold of class $C^1$, and let $\Phi^t$ be a $C^1$ flow on $M$. $\Phi^t$ is called a {\em gradient-like} flow if there exists a function $G:M\to\R$ such that,  for any $x\in M$, either $G(\Phi^t(x))<G(\Phi^s(x))$ for all $0\leqslant t<s$ or $\Phi^t(x)=x$ for all $t\geqslant 0$. Such a function $G$ is called a {\em Lyapunov function}. A point $x\in M$ is said to be a {\em rest point} of $\Phi^t$ if the orbit through $x$ is constant and we shall denote by $\text{Rest}(\Phi^t)$ the set of all rest points of the flow.

The following definition of relative Lusternik-Schnirelmann category is due to \cite{Cornea_book}. Let $X$ be a topological space and $A\subset X$. The {\em relative Lusternik-Schnirelmann category} of the pair $(X,A)$, denoted by $\text{Cat}(X,A)$, is the least integer $n\geqslant 1$ such that there exist open sets $U_0,U_1,\ldots,U_n$ in $X$, with $A\subset U_0$ and $X\subset\cup_i U_i$, such that, for all $i\geqslant1$, the set $U_i$ are contractible in $X$ and, for $i=0$, there exists a homotopy of pairs\footnote{Recall that for any pair of topological spaces $X$ and $Y$, and $A\subset X$, $B\subset Y$, $(X,A)$ and $(Y,B)$ are called {\em pairs of spaces}. A {\em map of pairs} $f:(X,A)\to(Y,B)$ is just a map $f:X\to Y$ such that $f(A)\subset B$. Two maps of pairs $f,g:(X,A)\to(Y,B)$ are homotopic if there is a homotopy $F$ with the additional restriction that $F(A\times[0,1])\subset B$.} $H:(U_0\times[0,1], A\times[0,1])\to (X,A)$ with $H_0$ the inclusion $U_0\hookrightarrow X$ and $H_1(U_0)\subset A$. It is clear that
$$
\text{Cat}(X,\varnothing)=\text{Cat}(X)
$$
where $\text{Cat}(X)$ denotes the classical Lusternik-Schnirelmann category of $X$. 
%\begin{Pro}\label{Pro_LS_1}
%For any pair $(X,A)$ we have
%$$
%\text{Cat}(X)+1\geqslant\text{Cat}(X,A)\geqslant\text{Cat}(X/A),
%$$
%where $\text{Cat}(X)$ is the classical Lusternik-Schnirelmann category with any topological space $X$, and $X/A$ is the quotient space of $X$ with respect to $A\subset X$.
%\end{Pro}

It is well known that the main interest of the Lusternik-Schnirelmann category comes from the fact that, for any smooth compact manifold $M$,  $\text{Cat}(M)$ gives a lower bound for the number of critical points of any smooth function $f$ on $M$.

Analogously, let $N\subset M$ be a compact topological submanifold of dimension $n$ of $M$ such that $N$ has a smooth interior and $\partial N=A\cup_{\partial A}B$ with $A$ and $B$ smooth $(n-1)$-dimensional submanifolds of $M$ such that $A\cap B=\partial A=\partial B$. Let $\Phi^t$ be a gradient-like flow on $M$ and let $W$ be the corresponding vector field. Assume that $W$ points out of $N$ on $A$ and inside $N$ on $B$. Denoting by $\text{Rest}_N(\Phi^t)=\text{Rest}(\Phi^t)\cap N$, we have the following.

\begin{Pro}[\cite{Cornea_book}]\label{Pro_LS_2}
Suppose $N$ and $\Phi^t$ are as above, then
$$
\text{Cat}(N,A)\leqslant\text{Rest}_N(\Phi^t).
$$
\end{Pro}

%{\bf WE HAVE REVISED THE MANUSCRIPT UP TO THIS POINT}

\begin{Rem}
It is worth noting that the setting we use here is similar to the index pair $(N,A)$ in the theory of Conley index, where $N$ is an {\em isolating neighborhood} and $A$ is the {\em exit set} for $N$. The only difference is that we use the complement $M\setminus N$  instead of $N$ for our purpose and in this case, if $\Phi^t$ is a gradient flow, then the exit set $A$ is empty.
\end{Rem}

Now we apply Proposition \ref{Pro_LS_2} to our case under condition \eqref{generic_condition} which, as we recalled above, holds true for a generic family of Tonelly Hamiltonians.

\begin{The}\label{LS_category}
Let $L$ be a Tonelli Lagrangian on $\T^n$ and, for any $c\in\R^n$, assume condition \eqref{generic_condition} so that the  barrier function takes the form $B^*_c(x)=u^-_c(x)-u^+_c(x)$. Then there exist at least $\text{Cat}(\T^n\setminus U)$ critical points of $B^*_c$ outside $\mathscr{A}_c$, where $U\supset\mathscr{A}_c$ is any sufficiently small open neighborhood of $\mathscr{A}_c$.
\end{The}

\begin{proof}
For any $\lambda>0$ small enough, $B_{\lambda}$ is of class $C^{1,1}$ and has the same critical points as $B^*_c$ by (P4) of Proposition \ref{same_critical}. Thus, it is enough to estimate the number of critical points of $B_{\lambda}$ outside $\mathscr{A}_c$ .

In order to give a lower bound for the number of critical points of $B_{\lambda}$  outside $\mathscr{A}_c$,  let us suppose  there exists an open set $U\supset\mathscr{A}_c$, such that $\mathscr{A}_c$ is the unique critical set (minimizers of $B_{\lambda}$) of $B_{\lambda}$ in $U$. This assumption can be made without loss of generality for, otherwise, $B_{\lambda}$ would have infinitely many critical points outside $\mathscr{A}_c$ and the conclusion would hold a fortiori. More precisely, we can assume that there exists $a_0>0$ such that any $a\in (0,a_0]$ is a regular value of $B_{\lambda}$, and take $U=U_a=\{y:B_{\lambda}(y)<a\}$ for some fixed $a\in (0,a_0]$. Observe that $U$ is an isolated invariant set of the gradient flow $\Phi^t_{\lambda}$ ($\lambda>0$ small enough) generated by the potential function $B_{\lambda}$, that is, 
$\Phi_{\lambda}^t(x)=x_{\lambda}(t)$, $t\in\R$, where
$$
\dot{x}_{\lambda}(t)=DB_{\lambda}(x_{\lambda}(t)).
$$
We can 
now apply Proposition \ref{Pro_LS_2} taking $N=\T^n\setminus U$ and $A=\varnothing$ because $\Phi_{\lambda}^t$ is a gradient flow and $B_{\lambda}$ is the required Lyapunov function. It follows that
$$
\text{Rest}_{\T^n\setminus U}(\Phi^t_{\lambda})\geqslant\text{Cat}(\T^n\setminus U,\varnothing)=\text{Cat}(\T^n\setminus U).
$$
Then $u_{\lambda}$ has at least $\text{Cat}(\T^n\setminus U)$ critical points outside $\mathscr{A}_c$.
\end{proof}

\subsection{Homoclinic orbits outside the Aubry set}
In this paper, a homoclinic orbit $(\gamma,\dot{\gamma}):(-\infty,+\infty)\to T\T^n$ (with respect to the Aubry set $\tilde{\mathscr{A}}_c$) is said to be {\em minimal} if there exists $t_0\in\R$ such that $\gamma$ is both backward calibrated  on $(-\infty,t_0]$ and  forward  calibrated  on $[t_0,+\infty)$.

It is clear that, when $\gamma:(-\infty,+\infty)\to\T^n$ produces a minimal homoclinic orbit with respect to Aubry set, there exists $t_0\in\R$ such that $x_0=\gamma(t_0)$ is a critical point of the barrier function $B^*_c$. On the other hand, for any critical point of $B^*_c$ outside $\mathscr{A}_c$, we cannot conclude whether it determines an expected minimal homoclinic orbit until verifying any conditions in Theorem \ref{criterion_I}.

Unfortunately, checking the validity of the conditions of Theorem \ref{criterion_I} may be difficult in arbitrary dimension $n\geqslant 2$ without any any assumption on $\mathscr{A}_c$. On the other hand, conditions c) and d) seem easier to handle in  dimension two because, recalling Remark \ref{re:MP},  it suffices to find critical points of the approximating barrier function $B_{\lambda}$ outside $\mathscr{A}_c$ which are not  local maximum points.

For our purposes, we need the following result by Hofer. Let $u\in C^1(\R^n)$ and let $x\in\R^n$ be a critical point of $u$. $x$ is called a critical point of {\em mountain-pass type} if, for any open neighbourhood $U$ of $x$, $u^{-1}((-\infty,u(x)))\cap U$ is nonempty and  not pathwise connected.

\begin{Pro}\cite{Hofer1985}\label{Hofer_result}
Let $u\in C^1(\R^n)$ and assume that $x_0, x_1\in\R^n$ are distinct points. Define
\begin{equation}\label{eq:b_def}
b=\inf_{\gamma\in\Gamma}\sup_{t\in[0,1]}u(\gamma(t)),
\end{equation}
where $\Gamma$ is the set of all continuous paths $\gamma:[0,1]\to\R^n$ with $\gamma(0)=x_0$ and $\gamma(1)=x_1$. If 
\begin{equation}\label{eq:b}
b>\max\{u(x_0), u(x_1)\}, 
\end{equation}
then there exists at least one critical point, with critical value $b$, which is either a local minimum point or a point of mountain-pass type. 
\end{Pro}

\begin{Rem}\label{re:b}
Note that, in the above lemma, \eqref{eq:b} is satisfied if $x_0,x_1\in\R^n$ are distinct isolated local minimum points of $u$. Indeed, taking  closed disjoint balls $B_0$ and $B_1$ centered at $x_0$ and $x_1$, respectively, let
$$
b_0=\min_{x\in\partial B_0}u(x)>u(x_0),\qquad b_1=\min_{x\in\partial B_1}u(x)>u(x_1).
$$
Then, by \eqref{eq:b_def}, for any $\gamma\in\Gamma$, 
$$
\max_{t\in[0,1]}u(\gamma(t))\geqslant\max\{b_0,b_1\}.
$$
It follows that $b\geqslant\max\{b_0,b_1\}>\max\{u(x_0), u(x_1)\}$.
\end{Rem}

\begin{The}\label{main_thm}
Let $L$ be a Tonelli Lagrangian on $\T^2$ and, for any $c\in\R^n$, assume condition \eqref{generic_condition} so that $B^*_c(x)=u^-_c(x)-u^+_c(x)$. 

If there exists an open neighborhood $U\subset\T^2$ of $\mathscr{A}_c$ such that $\T^2\setminus U$ is non-contractible, then there exists a minimal homoclinic orbit with respect to the Aubry set $\tilde{\mathscr{A}}_c$ outside $\tilde{\mathscr{A}}_c$.
More precisely, there exists a $C^2$ curve $\gamma:(-\infty,\infty)\to\T^2$ which is an extremal of the associated Euler-Lagrange equation, such that the $\alpha$-limit and $\omega$-limit sets of $(\gamma,\dot{\gamma})$ belong to $\tilde{\mathscr{A}}_c$. Moreover, $\gamma$ is a backward calibrated curve on $(-\infty,0]$ and a forward calibrated curve on $[0.\infty)$.
\end{The}

\begin{proof}
Let $U$ be any neighborhood of $\mathscr{A}_c$. By Theorem \ref{LS_category}, there exist at least $\text{Cat}(\T^2\setminus U)$ critical points of $B^*_c$ in $\T^2\setminus U$, and we have $\text{Cat}(\T^2\setminus U)\geqslant 2$ since $\T^2\setminus U$ is not contractible. Thus, there exist at least two distinct critical points of $B_{\lambda}$ outside $U$, where $B_{\lambda}$ is the Lasry-Lions regularization of  $B^*_c$.

By (P5) in Proposition \ref{same_critical} we have  that $x$ is a local maximum point of $B^*_c$ if and only if it is also a local maximum point of $B_{\lambda}$ with $0<\lambda\leqslant\lambda_1\leqslant\lambda_0$. Now,  suppose  all the critical points of $B_{\lambda}$ in $\T^2\setminus U$ are isolated local maximum points. Otherwise, there would exist a critical point of $B_{\lambda}$ ($B^*_c$) which satisfies condition (c) (see Remark \ref{re:MP}) or (d) of Theorem \ref{criterion_I} yielding the existence of the expected homoclinic orbit. 
Since there exist at least two isolated local maximum points of $B_{\lambda}$, say $x_0$ and $x_1$, then by defining
$$
b_{\lambda}=\sup_{\gamma\in\Gamma}\inf_{t\in[0,1]}B_{\lambda}(\gamma(t))
$$
as in Proposition \ref{Hofer_result} (here we use $-u$ instead of $u$), together with Remark~\ref{re:b}, we have a third critical point $x_2$ with the critical value $b_{\lambda}$ which is local maximum or of mountain-pass type in the sense of Hofer. In the latter case, we have the expected conclusion by condition (c) and Remark \ref{re:MP}. In the former case, we have a third isolated local maximum point of $B_\lambda$, say $x_2$.
Inductively, we can construct a sequence of isolated local maximum point of $B_\lambda$ (thus, of $B^*_c$),  a subsequence of which should converge to a cluster point $\bar{x}$. This contradicts the assumption that all the critical points of $B_{\lambda}$ in $\T^2\setminus U$ are isolated local maximum points and completes the proof.
\end{proof}

Finally, we would like to point out that not only does our method apply  to construct homoclinic orbits with respect to Aubry sets but could be used to connect orbits between different Aubry classes under condition \eqref{generic_condition_2}, which ensures there exists finitely many Aubry classes in $\mathscr{A}_c$ and holds true, once again, for a generic family of Tonelly Hamiltonians. Suppose that, for a given $c\in\R^n$, there exist distinct Aubry classes $A_1$ and $A_2$ in $\mathscr{A}_c$ and define the barrier function
$$
B_{1,2}(x)=u^-_1(x)-u^+_2(x),\quad x\in\T^n,
$$
where $u^-_1$ (resp. $u^+_2$) is an elementary backward (resp. forward) viscosity solution associated with class $A_1$ (resp. $A_2$).
\begin{The}\label{connecting_Aubry_class}
Let $L$ be a Tonelli Lagrangian on $\T^2$ and assume condition \eqref{generic_condition_2}. Let $c\in\R^n$ and suppose there exists an open neighborhood $U\subset\T^2$ of $\mathscr{A}_c$ such that $\T^2\setminus U$ is non-contractible. Then there must exist a connecting orbit between any pair of distinct Aubry classes such that each orbit passes through a critical point of the associated barrier functions $B_{1,2}$ in $\T^2\setminus U$. Moreover, such a critical point is of mountain-pass type or a nonisolated local maximum point.
\end{The}

\begin{proof}
Since $u^-(x)=h_c(y,x)$ for some $y\in\mathscr{M}_c$, then for any $x\in\T^2$ and $p\in D^*u^-(x)$ there exists a unique $(u^-,L_c,\alpha(c))$ calibrated $C^1$ curve $\gamma:(-\infty,0]$ such that $\gamma(0)=x$ and $p=\frac{\partial L_c}{\partial v}(\gamma(0),\dot{\gamma}(0))$. Denote by $y$ the projection of an $\alpha$-limit point of  $(\gamma,\dot{\gamma})$ onto $\T^n$. Now, recall  Aubry classes are connected sets (see, e.g. \cite{Contreras-Paternain}) and each of them contains an ergodic component of $\mathscr{M}_c$. Hence, if we assume that there are only finitely many Aubry classes, the connected components of $\mathscr{A}_c$ are finite and must coincide with the Aubry classes. This implies the $\alpha$-limit set of $\gamma$  belongs to the Aubry class containing $y$.

From this point on, the proof of the existence of  connecting orbits between the Aubry classes $A_1$ and $A_2$ uses the same reasoning of the proof of Theorem \ref{main_thm}, applied to the barrier function $B_{1,2}.$
\end{proof}

\begin{Rem}
For the study of the existence of possible genuine heteroclinic orbits connecting two distinct Aubry sets $\mathscr{A}_{c_1}$ and $\mathscr{A}_{c_2}$ with $[c_1]\not=[c_2]$, we need introduce some other kind of barrier functions. Unlike the homiclinic case, we always need condition \eqref{generic_condition_2} to ensure the finiteness of the Aubry classes for $\mathscr{A}_{c_1}$ and $\mathscr{A}_{c_2}$.

Fix $c_1$ and $c_2$, suppose that there exists $i_1$ (resp. $i_2$) distinct Aubry classes $A_{1,1},\ldots,A_{1,i_1}$ in $\mathscr{A}_{c_1}$ (resp. $A_{2,1},\ldots,A_{2,i_2}$ in $\mathscr{A}_{c_2}$). Denote by $u^-_{c_1,j}$ (resp. $u^+_{c_2,k}$), $j=1,\ldots,i_1$ ($k=1,\ldots,i_2$) the elementary backward (resp. forward) viscosity solutions determined by $A_{1,j}$ (resp. $A_{2,k}$). We set
$$
v^-_j(x)=\langle c_1,x\rangle+u^-_{c_1,j}(x),\quad v^+_k(x)=\langle c_2,x\rangle+u^+_{c_2,k}(x)\quad x\in\R^n
$$
Now, define the associated barrier functions
\begin{equation}\label{B_jk}
B_{j,k}(x)=v^-_j(x)-v^+_k(x),\quad x\in\R^n,
\end{equation}
where $j=1,\ldots,n_1$ and $k=1,\ldots,n_2$. Recall that, if $x$ is a local minimum point of $B_{j,k}$, then there exists an expected connecting orbit between $A_{1,i}$ and $A_{2,j}$. It is worth noting that we  cannot ensure the existence of critical points of $B_{j,k}$, in general, because, unlike in the homoclinic case,  the barrier function $B_{j,k}$ is the sum of a $\T^n$-periodic function with a nonzero linear function. We will study this case in the future.
\end{Rem}

\begin{Rem}
Actually, the result in Theorem \ref{main_thm} holds true under certain more general assumptions. More precisely, given any conjugate pair of weak KAM solutions $(u^-,u^+)$, define
$$
B(x)=u^-(x)-u^+(x).
$$
Then, by the same argument of the proof of Theorem \ref{main_thm} on $B_c^*$, under  condition \eqref{generic_condition} one can prove  the existence of minimal homoclinic orbits outside the Aubry set. The only difference is that, under condition \eqref{generic_condition} or  \eqref{generic_condition_2}, we can even determine an Aubry class as a specific $\alpha$- or $\omega$-limit sets according to a fixed conjugate pair of elementary weak KAM solutions as explained in the proof of Theorem~\ref{connecting_Aubry_class}. 
The connecting orbits between distinct Aubry classes provided by such a theorem pass through a critical point of mountain-pass type or an nonisolated local maximum point of the barrier function $B_{1,2}$ outside the Aubry set, unlike the ones in \cite{Contreras-Paternain,Fathi98} which are constructed by the Ma\~n\'e set in finite covering spaces.
\end{Rem}

\end{document}